\documentclass{amsart}

\usepackage[usenames,dvipsnames]{xcolor}
\usepackage{graphicx}

\usepackage{textcomp}
\usepackage{amsfonts}
\usepackage{amssymb}
\usepackage{amsthm}
\usepackage{stmaryrd}

\newcommand{\integers}{\mathbb{Z}}

\newtheorem*{corollary*}{Corollary}
\newtheorem*{theorem*}{Theorem}
\newtheorem{theorem}{Theorem}[section]
\newtheorem{lemma}[theorem]{Lemma}

\newtheorem{proposition}[theorem]{Proposition}
\newtheorem{corollary}[theorem]{Corollary}
\newtheorem{observation}[theorem]{Observation}

\theoremstyle{definition}
\newtheorem{definition}[theorem]{Definition}

\newtheorem*{example*}{Example}
\newtheorem*{convention*}{Convention}

\theoremstyle{remark}

\numberwithin{equation}{section}

\newcommand{\dl}{\text{DL}}

\begin{document}

\renewcommand{\bf}{\bfseries}
\renewcommand{\sc}{\scshape}
\vspace{0.5in}

\title{Visual Boundaries of Diestel-Leader Graphs}

\author{Keith Jones}
\address{Department of Mathematics, Computer Science \& Statistics; State University of New York College at Oneonta; 108 Ravine Parkway; Oneonta, NY 13820}
\email{keith.jones@oneonta.edu}

\author{Gregory A. Kelsey}
\address{Department of Mathematics; Trinity College; 300 Summit St.; Hartford, CT 06106}
\email{gregory.kelsey@trincoll.edu}



\thanks{The authors thank Moon Duchin and Melanie Stein for their
helpful conversations. The authors appreciate the helpful comments from the anonymous referee, 
particularly those which improved the approach in Section \ref{dl2qtop-section}. 
Support for the first author from a summer research grant from
SUNY Oneonta is gratefully acknowledged. 
Both authors also thank the Institute for Advanced Study for its hospitality at the 
Park City Math Institute Summer Session 2012.}

\begin{abstract}
Diestel-Leader graphs are neither hyperbolic nor CAT(0), so their
visual boundaries may be pathological. Indeed, we show that for $d>2$,
$\partial\dl_d(q)$ carries the indiscrete topology.  On the
other hand, $\partial\dl_2(q)$, while not Hausdorff, is $T_1$, totally
disconnected, and compact.  Since $\dl_2(q)$ is a Cayley graph of the
lamplighter group $L_q$, we also obtain a nice description of
$\partial\dl_2(q)$ in terms of the lamp stand model of $L_q$ and
discuss the dynamics of the action.
\end{abstract}
\maketitle

\section{Introduction}
The {\em visual boundary} $\partial M$ of a complete CAT(0) metric
space $M$ is the topological space obtained by giving the set of asymptotic equivalence classes of geodesic
rays in $M$ the compact-open topology \cite[Ch. II.8]{bh}. For any base
point $p \in M$, one can simply take
$\partial_p M$ to be the set of geodesic rays emanating from $p$, and
$\partial_p M$ and $\partial M$ are homeomorphic.  In this
setting, the visual boundary has nice properties: for instance, $M\cup\partial M$ is
contractible, and if $M$ is proper\footnote{A {\em proper} metric
space is one in which every closed metric ball is compact.} then
$\partial M$ provides a compactification of $M$ under the ``Cone
topology''. An action of a group
$G$ by isometries on $M$ can be extended to an action by
homeomorphisms on $\partial M$, and studying the dynamics of this
action can prove quite fruitful.
One can define the visual boundary more generally (i.e., outside the
context of CAT(0) spaces), and ask whether 
these nice properties still arise or whether the study of the action
on the boundary is still fruitful.

When a group $G$ acts geometrically on a space, one may
take the boundary of the space as a boundary of the group.  For word hyperbolic
groups, the visual boundary is unique, and has proven very useful 
\cite{kapovichbenakli}. Outside this
class of groups, the situation is not so nice. Croke and Kleiner have
shown that even CAT(0) groups may not have unique visual boundaries
\cite{crokekleiner}.  Even worse, outside this context
we may run into pathological situations. 
In \cite{Z2boundary}, it is shown that the visual boundary of 
the Cayley graph of $\integers^2$ with respect to the standard
generating set is uncountable, yet it has the indiscrete (a.k.a. trivial) 
topology.  In short, this occurs because one is able to play the 
asymptotic classes (two rays are equivalent if they are close in the
long term) against the compact-open topology (two rays are close if
they agree in the short term) to obtain a sequence of asymptotic rays
representing an arbitrary point of the boundary and whose limit is
another arbitrary point of the boundary.  

This paper investigates whether visual boundaries
for non-hyperbolic, non-CAT(0) groups may carry interesting
topologies, and if so what this might tell us about those
groups. In particular, we study the family of lamplighter 
groups $L_q = \integers_q \wr \integers$, $q \geq 2$ an
integer. Using the appropriate generating set, one obtains a
particularly nice Cayley graph for $L_q$, called the 
Diestel-Leader
graph $\dl_2(q)$ \cite{horocyclic}, \cite[\S 2]{Woess-Lamplighter}.
The boundary $\partial \dl_2(q)$ is not a canonical boundary for
$L_q$, since using a different Cayley graph might give rise to a
different boundary. However, $\partial \dl_2(q)$ is appealing in that
it can 
be well understood using the standard ``lamp stand'' model for the lamplighter
group. This model and the (well-known) geometry of Diestel-Leader graphs provide
ample tools for studying the visual boundary.

More generally, the Diestel-Leader graph $\dl(q_1,q_2,\dots,q_d)$ can be
realized as a subspace of the product of $d$ trees, having respective valence
$q_1+1$, $q_2+1$, $\dots$, $q_d+1$ \cite{horocyclic}. The notation $\dl_d(q)$ is used
when each tree has valence $q+1$. 
While we will state our results only for the cases when the degrees are equal, 
allowing the degrees to vary between trees has no effect on our analysis.
Recently, Stein and Taback have described the metric on these graphs
\cite{steintaback}, and Duchin, Leli\`{e}vre, and Mooney have
discussed geodesics in $\dl_d(q)$ 
in their work on sprawl \cite{sprawl}. 
While not all Diestel-Leader graphs
are Cayley graphs \cite[Theorem 1.4]{eskinfisherwhyte}, \cite[Corollary 2.15]{horocyclic}, 
this paper discusses the relation to $L_2$ when applicable and also
has results that apply in the non-Cayley graph case.
Because Diestel-Leader graphs inherit much of their structure
from trees, which are prototypical CAT(0) and hyperbolic spaces, it seems
natural to ask whether the boundaries of such graphs inherit any nice
properties from the boundaries of trees. 

In Section
\ref{backgroundsection}, we provide some background on visual
boundaries, lamplighter groups, and Diestel-Leader graphs.
In Section \ref{dl2qset-section} we
collect some basic facts about geodesic rays in Diestel-Leader graphs
and prove:

\begin{theorem*}[A - Corollary 3.7]
As a set, $\partial \dl_2(q)$ is a disjoint union of two punctured
Cantor sets.
\end{theorem*}

In Section \ref{dl2qtop-section} we prove:

\begin{theorem*}[B] $\partial \dl_2(q)$ is not Hausdorff, but it is
$T_1$, compact, and totally disconnected.
\end{theorem*}
The proof of this Theorem is collected in Observations
\ref{nonhausdorffobs}, \ref{t1obs}, Proposition \ref{compactobs}, and
Observation \ref{disconnectedobs}. Additionally, we discuss the
dynamics of the action by $L_2$ on $\partial \dl_2(q)$ in
Theorem \ref{actionthm} and Corollary \ref{actioncor}.
In Section \ref{dldqsection} we discuss the geometry of geodesics in $\dl_d(q)$
and prove: 
\begin{theorem*}[C - Theorem \ref{indiscretethm}]
For $d>2$, the topology of $\partial \dl_d(q)$ is indiscrete.
\end{theorem*}

Roughly speaking, this is a consequence of
the additional degree of freedom a third tree provides. Along the way, 
we establish  in Theorem \ref{geodesicturnsthm} a strong restriction on the
kinds of paths in $\dl_d(q)$ which may be geodesics.

One feature of CAT(0) and hyperbolic spaces is that the
horofunction\footnote{In short, for a
metric space $X$, a horofunction is a point in $C(X)$ (the space of
continuous functions on $X$ with the topology of compact convergence
on bounded subsets) which is a limit of a sequence of functions
$d_y(x) = d_X(x,y)$, $y \in X$. Horofunctions represent points of the
horofunction bounday $\partial_h(X)$. 
Busemann functions are those horofunctions obtained as a limit of points along 
a geodesic ray in $X$. 
See \cite[II.8.12-14]{bh} for details.}
boundary $\partial_h X$ is naturally homeomorphic to $\partial X$
\cite[II.8.13]{bh}, since all horofunctions are Busemann
functions (i.e. they come from geodesic rays).  However, outside this
setting, one may find horofunctions which are not Busemann functions
\cite{websterwinchester}. An investigation of the horofunction
boundary will appear in a forthcoming work, where we will show that
while $\partial \dl_2(q)$ embeds in $\partial_h \dl_2(q)$, there are many
horofunctions which are not Busemann functions.

\section{Background}
\label{backgroundsection}

\subsection{The Visual Boundary.}
\label{visualboundariessub}
Let $X$ be a geodesic space with base point $x_0$. Two geodesic rays
$\gamma, \gamma' : [0,\infty) \rightarrow X$ are said to be {\em asymptotic} if
there is a $\lambda \geq 0$ such that for all $t\geq0$,
$d(\gamma(t),\gamma'(t)) \leq \lambda$. The {\em visual boundary} of $X$ is
the space $\partial X$ consisting of all asymptotic equivalence
classes of geodesic rays in $X$, endowed  with the quotient topology from
the topology of
uniform convergence on compact sets.  The {\em based visual bounday} of $X$
with base point $x_0$, denoted $\partial(X, x_0)$, 
is the same topology restricted to the subset of
geodesic rays emanating from $x_0$. In general the based and unbased
visual boundaries need not agree.

In this paper we consider (based) Diestel-Leader graphs, 
which are known to be vertex transitive
\cite[Proposition 2.4]{horocyclic}, so the based visual boundary is
independent of base point.
In Proposition
\ref{basepointprop}, we show that in $\dl_2(q)$, the based and unbased
boundaries are the same, so throughout Sections \ref{dl2qset-section} and
\ref{dl2qtop-section}, we abuse notation and use $\partial \dl_2(q)$ to refer
to the based visual boundary of $\dl_2(q)$.
In Section \ref{dldqsection}, we still abuse notation and  use 
$\partial \dl_d(q)$ to refer
to the based visual boundary, even though we do not consider whether it is the
same as the unbased visual boundary.

\subsection{The Diestel-Leader Graph $\dl_d(q)$.}
For an integer $q$, let $T$ be the regular $q+1$-valent
simplicial tree.
Following \cite[\S 2]{Woess-Lamplighter} and \cite[\S
2]{steintaback}, we orient the edges of $T$ so that each vertex $v$ has
exactly one {\em predecessor} $v^-$  and $q$ {\em successors}. This
induces a partial ordering on the set of vertices of $T$, under which
any two vertices $v$ and $w$ have a greatest common ancestor $v
\curlywedge w$. Choosing a base vertex $o$ in $T$ allows us to define
a height function $h(v) = d_T(v, v \curlywedge o) - d_T(o, v \curlywedge
o)$, where the function $d_T$ measures distance in $T$ when each edge
is given length 1. 
The partial ordering provides a chosen endpoint
$\omega$ of $T$, obtained by any geodesic ray that always follows
predecessors, and this height function is the Busemann function for
$\omega$ corresponding to the ray emanating from $o$. For a vertex $v$ in the
horocycle $H_k = \{ v \in T \mid h(v) = k \}$, its unique precedessor $v^-$
is in $H_{k-1}$, and each of its $q$ successors is in $H_{k+1}$ 
(see Figure \ref{dl22fig}).
In particular, for a given initial vertex $v$, there
is a unique ``downward'' path of length $k$, for each $k$, and a
unique downward ray: that which leads to $\omega$.

We now define $\dl_d(q)$. 
Let $T_0$, $T_1, \dots, T_{d-1}$ be copies of
$T$, with base points $o_i \in T_i$.  The Diestel-Leader graph
$\dl_d(q)$ is the graph whose vertex set consists of $d$-tuples $v=
(x_0, x_1, \dots, x_{d-1})$, $x_i \in T_i$ a vertex, such that
$\sum_{i=0}^{d-1} h(x_i) = 0$.  Let $h_i(v)$
denote the height function $h(x_i)$ on $T_i$.  There is a natural
basepoint $(o_0, o_1, \dots, o_{d-1})$ for $\dl_d(q)$.

The edges of $\dl_d(q)$ correspond to pairs  $(v, w) =
((x_0, ...,x_{d-1}),(y_0, ...,y_{d-1}))$ such that there are $i$ and
$j$, $i\neq j$ with an edge joining $x_i$ to $y_i$ in $T_i$, an
edge joining $x_j$ to $y_j$ in $T_j$, and $x_k=y_k$ for all $k\neq i,j$. The relation \[h(y_i) -
h(x_i) = \pm 1 = h(x_j) - h(y_j) \] follows from the definition of
vertices of $\dl_d(q)$.  Thus, moving along an edge in $\dl_d(q)$
means simultaneously choosing one tree in which to increase height, and
another tree in which to decrease height, while holding the position
constant in every other tree. 

\begin{figure}
\includegraphics{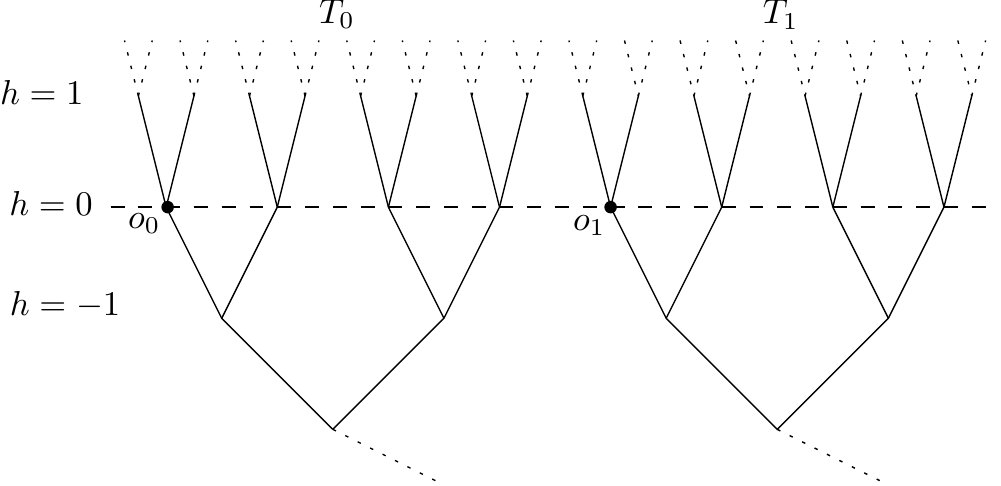}
\caption{A neighborhood of the origin in $\dl_2(2)$
}
\label{dl22fig}
\end{figure}

\begin{convention*} 
We adopt the convention that all geodesic rays in $\dl_d(q)$ under
discussion emanate from $o$, unless otherwise stated.
\end{convention*}

\subsection{Lamplighter groups}
\label{lamplighterbackgroundsubsection}

The Diestel-Leader graph $\dl_2(q)$ is the Cayley graph of the lamplighter group $L_q = \mathbb{Z}_q\wr\mathbb{Z}$ with generating set $\{t, at, a^2t, ... , a^{q-1}t \}$ ($t$ is the generator of $\mathbb{Z}$ in the wreath product and $a$ is the generator of $\mathbb{Z}_q$). Each element of the lamplighter group (and thus each vertex of $\dl_2(q)$) is associated with a ``lamp stand.'' In the case $q=2$, the lamp stand consists of a row of lamps in bijective correspondence with $\mathbb{Z}$, a finite number of which are lit, and a lamplighter positioned at one of the lamps. If $q>2$, then the lamps have $q$ settings: these can be interpreted as off and $q-1$ levels of brightness while lit or $q-1$ different colors \cite{parry}. The Diestel-Leader graph $\dl_3(q)$ has a similar interpretation, except that there is a rhombic grid of lamps \cite{clearyriley}.

We should also note that for $d>3$, if $q$ has a prime factor $p$ such that $p< d-1$, it is open whether $\dl_d(q)$ is a Cayley graph of some group. If $q$ has no such prime factor, then $\dl_d(q)$ is a Cayley graph \cite[Corollary 3.17]{horocyclic}. If $q$ does have such a prime factor, it is only known that $\dl_d(q)$ is quasi-isometric to a Cayley graph \cite[Corollary 3.21]{horocyclic}.

In the $d=2$ case, the base vertex $(o_0, o_{1})$ corresponds to the lamp stand
with no lit lamps and the lamplighter at position 0. In general, the
lamp stand corresponding to vertex $(x_0, x_{1})$ has the
lamplighter at position $h(x_0)$.  An edge in $\dl_2(q)$ corresponds to the lamplighter stepping between adjacent lamps. If the edge is associated with generator $t$, then the lamplighter moves without switching any bulbs. If the edge is associated with generator $at$, then the lamplighter switches the bulb before he leaves (if he is moving in the positive direction) or after he arrives (if he is moving in the negative direction) \cite{Woess-Lamplighter}.

So, for example, the word $t^3(at)t^{-2}(at)^{-2}t^{-1}$ corresponds with the lamplighter starting at position 0, moving three to the right, lighting lamp 3 and stepping to position 4, stepping back to position 2, then stepping back to light lamps 1 and 0, and finally stepping back to lamp $-1$. The end result is the lamp stand pictured in Figure \ref{lampstandexamplefig}. Notice that this lamp stand is also obtainable by the word $(at)^2t(at)t^{-5}$, and so these words represent the same element of $L_2$.

\begin{figure}[htbp]
\begin{center}
\includegraphics[width=4in]{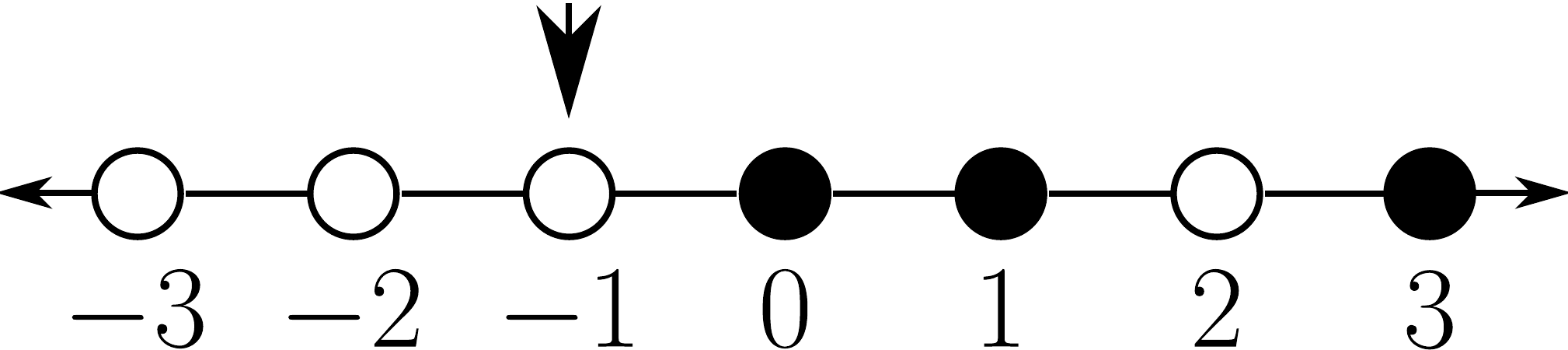}
\caption{Lamp stand example for $t^3(at)t^{-2}(at)^{-2}t^{-1}$}
\label{lampstandexamplefig}
\end{center}
\end{figure}

Multiplication of elements corresponds to ``composition'' of lamp stands. To compute the lamp stand for the element $g\cdot h$, take the lamp stand for $g$, then have the lamplighter perform the same switching of lamps as in $h$, but starting from the lamplighter's end position in $g$ instead of at $0$. So, for example, for $g$ the lamp stand in Figure \ref{lampstandexamplefig}, $g\cdot t$ would have the same set of lit lamps, but the lamplighter would be at $0$ instead of $-1$. The lamp stand for $g\cdot tat$ would have the lamplighter at position $1$ and lit lamps only at positions $1$ and $3$.

Note that $L_q$ has presentation
\[ \langle a,t \mid [t^i at^{-i}, t^j at^{-j}] = a^q = 1 \text{ for all }
i,j \in \integers\rangle, \]
from which we obtain the epimorphism $exp_t: L_2 \rightarrow \integers$
mapping $a\mapsto 0$ and $t \mapsto 1$, recording the exponent sum of
$t$ for a given element of $L_2$.

\section{The boundary of $\dl_2(q)$}
\label{dl2qset-section}

\subsection{Using projections to bound distance}

We begin with two observations that apply in the general case.
Since a path in $\dl_d(q)$ projects to a path in a tree $T_i$, we have a lower
bound on the distance between two vertices:
\begin{observation}
\label{lowerboundondistobs}
Let $v = (v_0, v_1, \dots, v_{d-1})$ and $w = (w_0, w_1, \dots,
w_{d-1})$ be two vertices in $\dl_d(q)$. 
Then $d(v,w) \geq \max\{d_{T_i}(v_i,w_i)\, |\, 0 \leq i \leq d-1\}$.
\end{observation}

Moreover, there is a simple upper bound on the distance as well:
\begin{lemma}
\label{upperboundondistobs}
Let $v = (v_0, v_1, \dots, v_{d-1})$ and $w = (w_0, w_1, \dots, w_{d-1})$ be
two vertices in $\dl_d(q)$. Then
\[ d(v,w) \leq \sum_{i=0}^{d-1} d_{T_i}(v_i, w_i) \]
\end{lemma}
\begin{proof}
We may assume $v$ is the origin $o$, since any path between vertices may
be translated via isometry to a path from the origin.
We will now construct a path in $\dl_d(q)$ from $o$ to $w$ that has length at most 
$ \sum_{i=0}^{d-1} d_{T_i}(o_i, w_i)$.

Let $k$ be the index of a tree such that $h_k(w)$ is minimal. Then
$h_k(w) \leq 0$. For each tree $T_i$ other than $T_k$, in turn, follow the
path from $o_i$ to $w_i$, 
always compensating in tree $T_k$ (i.e. moving up in $T_k$ when moving down
in $T_i$, and vice versa),  following the rule that when the
current vertex in $T_k$ has negative height and we must move up in
$T_k$, we choose to stay on the ray from $o_k$ to the distinguished
end $\omega_k$.
After all trees other than
$T_k$ have been so traversed, a total distance of $\sum_{i=0,i\neq
k}^{d-1} d_{T_i}(o_i, w_i)$ has been traveled. At this point, the
current vertex $x \in T_k$ lies
at height $h_k(w)$ and along the ray from $o_k$ to $\omega_k$, 
implying that $x$ lies on the geodesic from $o_k$ to $w_k$; and
so
$d_{T_k}(x, w_k) \leq d_{T_k}(o_k, w_k)$. Thus we can move from $o$
to $w$ taking no more than $\sum_{i=0}^{d-1} d_{T_i}(o_i,w_i)$ steps. 
We compensate for the steps moving $T_k$ into position in one other tree.
 Those compensating steps will undo themselves, and the vertex in that
tree at the end will be the same as it was before.
\end{proof}

\subsection{Asymptotic equivalence classes}
\label{asymptoticclassessection}

\begin{definition}
\label{projectiondef}
We can project a path $\gamma = (v_0, v_1, \dots, v_n)$ through
$\dl_d(q)$ to a tree $T_j$. Let 
$v_i = (x_{i,0}, x_{i,1}, \dots, x_{i,{d-1}})$ be the $i^{th}$ vertex along
$\gamma$. Then $\gamma^{(j)}$ is the corresponding path  
$(x_{0,j}, x_{1,j}, \dots, x_{n,j})$ through
$T_j$.  We will refer to $\gamma^{(j)}$ as the \emph{projection} 
of $\gamma$ to tree $T_j$
\end{definition}

For the rest of this section, we restrict our attention to 
$\dl_2(q)$, though the
notation and terminology we use will extend to the general case.

Let $p$ be a path between vertices $v$ and $w$ in $\dl_2(q)$. 
Thinking of $p$ as a sequence $(v
= v_0,  v_1, \dots, v_n = w)$ of vertices of
$\dl_2(q)$, we refer to a
subsequence $v_i,  v_{i+1}, 
v_{i+2}$ such that $h_j(v_i) = h_j(v_{i+1}) + 1 = h_j(v_{i+2}) $ as a {\em
turn} in $p$; a {\em bottoming out} in the $j^{th}$ tree. In the lamp
stand interpretation for $d=2$, bottoming out corresponds with the
lamplighter turning around and moving in the opposite direction along
the row of lamps: 
if the path bottoms out in $T_0$, then the lamplighter stops moving to the left (towards $-\infty$) 
along the lamp stand and begins moving to the right, and vice versa for $T_1$.

In \cite[Figure 1]{sprawl} it is shown that a geodesic in $\dl_2(q)$
has at most two turns.  However, the case of geodesic rays is
simpler.

\begin{lemma}
There are no 2-turn geodesic rays in $\dl_2(q)$. 
\label{no2turnraysobs}
\end{lemma}
\begin{proof}
Let $\gamma$ be a ray in $\dl_2(q)$ emanating from $o$ with two turns
and no back-tracking.  Then $\gamma$ ``bottoms out'' once in $T_i$,
$i\in\{0,1\}$, and then once in $T_{1-i}$.  Let $k \in \integers$, $k > 0$
be the distance traveled before the first turn, so that $\gamma$
bottoms out in $T_i$ at $\gamma(k)$ with heights $h_i(\gamma(k)) = -k$
and $h_{1-i}(\gamma(k)) = k$.  Let $l \in \integers$, $l > 0$, be the
distance traveled before the second turn, so that $\gamma$ bottoms out
in $T_{1-i}$ at $\gamma(k+l)$ with heights $h_i(\gamma(k+l)) = l-k$ and
$h_{1-i}(\gamma(k+l)) = k-l$. Because $\gamma$ has exactly two turns,
$\gamma$ then proceeds to descend eternally in $T_i$ while ascending
in $T_{1-i}$.  Consider the vertex $\gamma(k+2l)$. In $T_i$,
$\gamma$ descends to $\gamma^{(i)}(k)$, then ascends for $l$ edges
followed by a descent of the same distance, so that $\gamma^{(i)}(k) =
\gamma^{(i)}(k + 2l)$. 
 
We now show there is a path through $\dl_2(q)$ from $o$ that arrives
at $\gamma(k+2l)$ in shorter time. It may be helpful to refer to
Figure \ref{dlreplace2turnfig}, which provides an example.
Let $d$ be the distance in
$T_{1-i}$ from $o_{1-i}$ to $\gamma^{(1-i)}(k+2l)$. Then either $d= k$ if
$k \geq l$ or  $d=2l-k$ if $l > k$.  Choose a path $\gamma'$ from
$o$ to $\gamma(k+2l)$ so that $\gamma'^{(1-i)}$ is the geodesic in
$T_{1-i}$ from $o_{1-i}$ to $\gamma^{(1-i)}(k+2l)$, and so that
$\gamma^{(i)}$ changes height appropriately with each edge in
$\gamma^{(1-i)}$. The actual choice of $\gamma^{(i)}$ is irrelevant:
regardless of how it
ascends, it must then descend to $\gamma^{(i)}(k+2l) =
\gamma^{(i)}(k)$.  Thus
$\gamma'(d) = \gamma(k+2l)$ and $d < k+2l$, so $\gamma$ cannot be
a geodesic ray. 
\end{proof}

It is worth noting that because there are 2-turn geodesic paths in $\dl_2(q)$,
Lemma \ref{no2turnraysobs}
implies that $\dl_2(q)$ is not {\em geodesically complete}, i.e.
there are geodesics which cannot be extended to geodesic rays. This
fact is proved in \cite[Theorem 12]{steintaback}, where the authors demonstrate
that $\dl_d(q)$ has {\em dead-end elements}.

\begin{figure}
\includegraphics{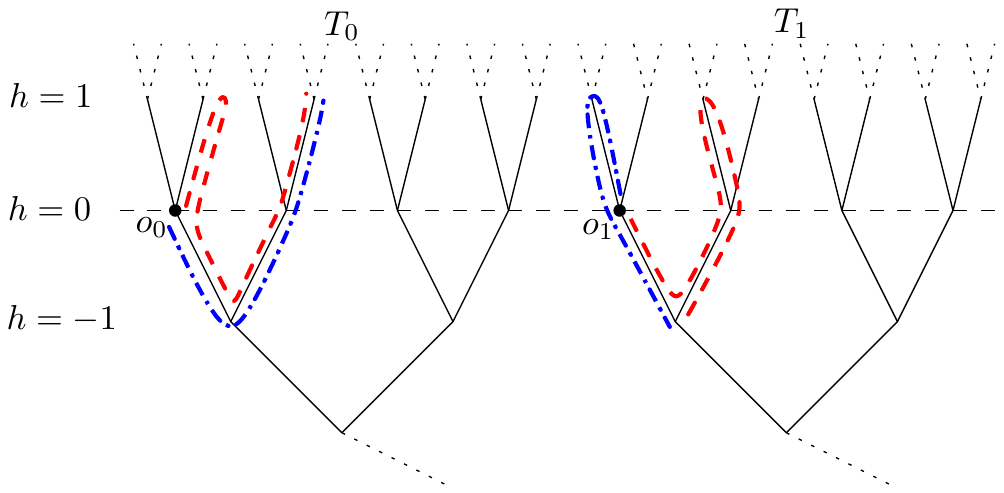}
\caption{Two paths in $\dl(2,2)$ having the same endpoints. A 1-turn
path which is geodesic, and a 2-turn path which is minimally
non-geodesic.}
\label{dlreplace2turnfig}
\end{figure}

Let $\gamma$ be a geodesic ray based at the origin vertex $o \in \dl_2(q)$,
and suppose $\gamma$  
first descends to height $-h \leq 0$ in $T_i$ before ascending 
eternally in $T_i$. Then $\gamma^{(i)}|_{[0,h]}$ is fixed, while
$\gamma^{(1-i)}|_{[0,h]}$ follows one of $q^h$ paths. On the other hand,
$\gamma^{(i)}|_{[h,\infty]}$ chooses some endpoint other than $\omega_i$
of $T_i$, while $\gamma^{(1-i)}|_{[h,\infty]}$ must approach
$\omega_{1-i}$.

\begin{lemma}
\label{sameheightraysobs}
If $\gamma$ and $\gamma'$ are geodesic rays in $\dl_2(q)$ emanating from 
$o$ which both descend to
height $-h \leq 0$ in $T_i$ before turning, then they are asymptotic if
and only if $\gamma^{(i)} = \gamma'^{(i)}$. 
In this case, they eventually merge and upon merging, never split.
\end{lemma}

\begin{proof}
First, assume that $\gamma^{(i)} = \gamma'^{(i)}$. Because
$h_{1-i}(\gamma^{(1-i)}(h)) = h_{1-i}(\gamma'^{(1-i)}(h)) = h$, and both
of these vertices  have $o_{1-i}$ as an ancestor, we have $\gamma^{(1-i)}(2h) =
\gamma'^{(1-i)}(2h) = o_{1-i}$. From this point on, since both
$\gamma^{(1-i)}$ and $\gamma'^{(1-i)}$ are descending in $T_{1-i}$, they
are the same.   Hence, $\gamma|_{[2h,\infty)} = \gamma'|_{[2h,\infty)}$. 
So $\gamma$ and $\gamma'$ merge at or before
distance $2h$, and they are asymptotic.

Now, assume that there exists $t$ such that $\gamma^{(i)}(t) \neq
\gamma'^{(i)}(t)$. Since both rays descend to height $-h$, we must have $t>h$. 
Since $\gamma$ and $\gamma'$ are geodesic rays bottoming out at height
$h$ in $T_i$, it follows that $\gamma^{(i)}|_{[h,\infty)}$ and
$\gamma'^{(i)}|_{[h,\infty)}$ are geodesic rays in $T_i$. By assumption
they are not equal, so since $T_i$ is a tree, they are not asymptotic
in $T_i$.  From Observation \ref{lowerboundondistobs}, $\gamma$ and
$\gamma'$ are not asymptotic.
\end{proof}

Furthermore, Observation \ref{lowerboundondistobs} also ensures the
following:

\begin{observation}
\label{differentheightraysobs}
Let $\gamma$ and $\gamma'$ be geodesic rays in $\dl_2(q)$. 
\begin{enumerate}
\item If $\gamma$ begins by descending in $T_i$, while $\gamma'$
begins by ascending in $T_i$, or vice-versa, then $\gamma$ and
$\gamma'$ are not asymptotic.

\item If $\gamma$ and $\gamma'$ are geodesic rays descending to 
heights $h$ and $h'$ respectively before turning, with $h \neq h'$, then 
$\gamma$ and $\gamma'$ are not asymptotic. 
\end{enumerate}
\end{observation}

Combining Lemmas \ref{no2turnraysobs}, \ref{sameheightraysobs} and Observation
\ref{differentheightraysobs}, we obtain the following description of 
asymptotic equivalence classes of geodesic rays in $\partial\dl_2(q)$:

\begin{theorem}
\label{asymptoticthm}

Two geodesic rays $\gamma$ and $\gamma'$ in $\dl_2(q)$ are asymptotic
if and only if their projections $\gamma^{(0)}, \gamma'^{(0)}$
approach the same end of $T_0$ and their projections $\gamma^{(1)},
\gamma'^{(1)}$ approach the same end of $T_1$.

\end{theorem}

\begin{corollary}
\label{setcor}

The family of geodesic rays whose projections do not approach $\omega_i \in
\partial T_i$, $i\in\{0,1\}$, is in one-to-one correspondence with a
Cantor set minus the point corresponding to $\omega_i$.  Hence, as a
set, $\partial \dl_2(q)$ is a disjoint union of two deleted Cantor
sets: 
\[ \partial \dl_2(q) = \left(\left(\partial T_0 - \omega_0\right) \times
\{\omega_1\}\right) \coprod \left(\left(\partial T_1 - \omega_1\right) \times
\{\omega_0\}\right) \] 

\end{corollary}

It is perhaps not surprising that $\partial \dl_2(q)$ should be so closely
related to a Cantor set, given that $\dl_2(q)$ is a one dimensional
subset of a product of trees.
Lemma \ref{sameheightraysobs} leads to the picture in Figure
\ref{asymptoticclassfig} of a typical element of $\partial \dl_2(2)$.

\begin{figure}
\includegraphics[scale=.55]{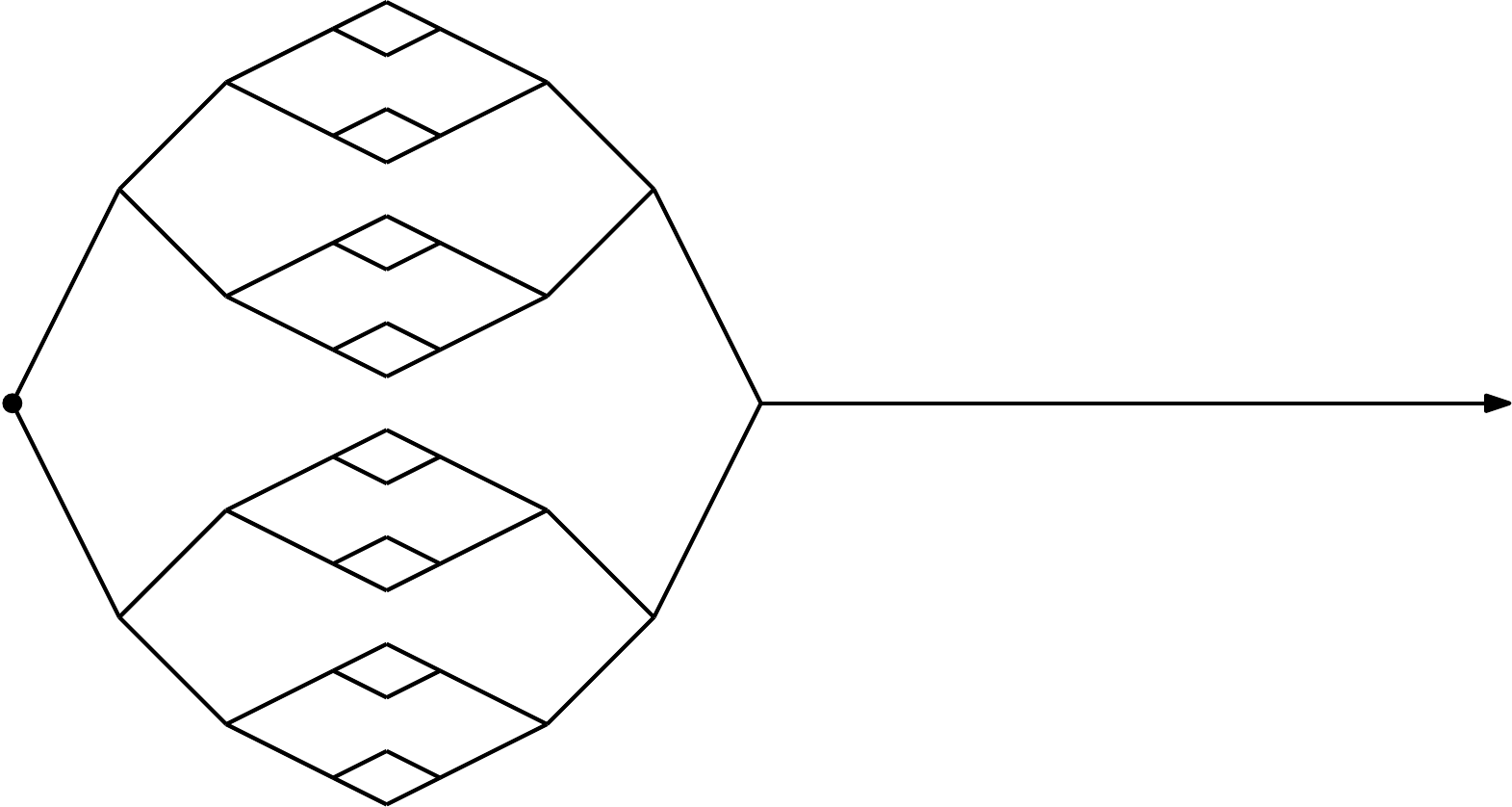}
\caption{A typical asymptotic equivalence class in $\dl_2(2)$.}
\label{asymptoticclassfig}
\end{figure}

\subsection{Lamp stand interpretation of $\partial\dl_2(2)$}
\label{lampstandsub}

We can understand the visual boundary using the lamp stand. For each geodesic ray starting at the base point, the lamplighter starts at position 0 on an unlit row of lamps. He starts moving in a direction (always the same direction as the projection of the ray to $T_0$), perhaps lighting lamps along the way. If the ray ``bottoms out'' in one tree, then the lamplighter will turn around and proceed in the other direction, again possibly switching lamps along the way. So each element of the visual boundary corresponds with a lamp stand with the lamplighter standing at either $+\infty$ or $-\infty$. If the lamplighter is at $+\infty$, then the set of lit lamps (if it is non-empty) has a minimum. If the lamplighter is at $-\infty$, then the set of lit lamps (if it is non-empty) has a maximum.

Notice that if the lamplighter turns, he can reset the lamps he has already passed to undo any lighting that he has done or to light any lamps that he missed the first time. In this way, we can see how the ``pre-turn'' segment of the ray does not affect the asymptotic equivalence class.

Since the height of the associated vertex in tree $T_0$ is the position of the lamplighter, this means that the points in $(\partial T_0 - \omega_0)\times \omega_1$ have the lamplighter at $+\infty$ and the points in $(\partial T_1 - \omega_1)\times \omega_0$ have the lamplighter at $-\infty$.

The lamp stand interpretation for $\partial\dl_2(q)$ is essentially the same, except that the lamps can take on $q$ different states, instead of simply on and off.

\subsection{Action of $L_2$ on $\partial\dl_2(2)$}
\label{actionsub}

We can compute the action of the lamplighter group $L_2$ on the visual boundary $\partial\dl_2(q)$ by using the lamp stand interpretation in Section \ref{lampstandsub}. For $\gamma$ a geodesic ray in $\dl_2(q)$, we write $[\gamma]$ for its asymptotic equivalence class in $\partial\dl_2(q)$. For $g\in L_2$ and $[\gamma]\in \partial\dl_2(2)$, to compute the lamp stand for $g\cdot[\gamma]$, start with the lamp stand for $g$. Then, have the lamplighter perform the lighting prescribed by $[\gamma]$, but starting from the lamp lighter's end position in $g$ instead of at position 0. See Figure \ref{actionfig} for an example.

\begin{figure}[htbp]
\begin{center}
\includegraphics[width=3.5in]{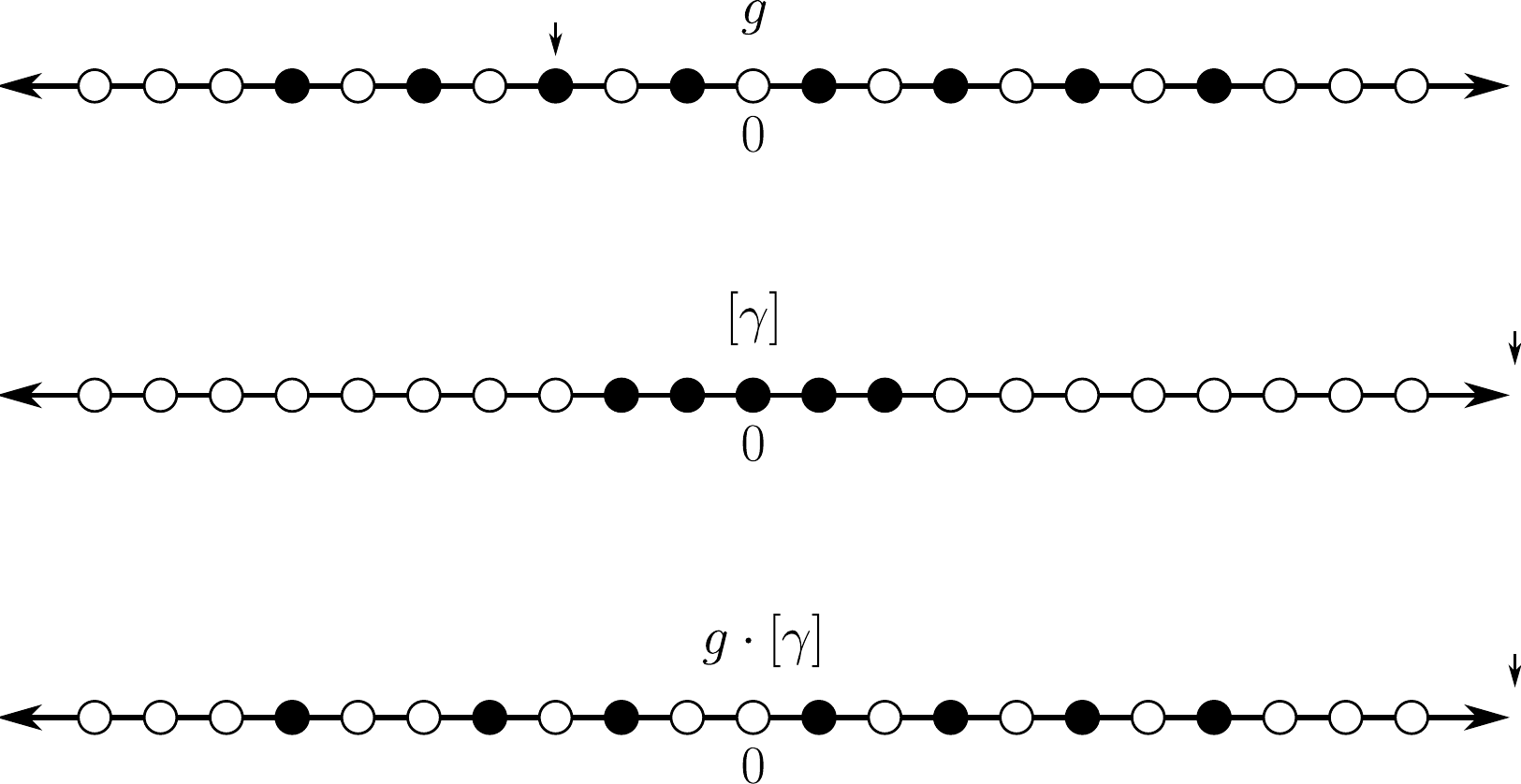}
\caption{The action of an element $g\in L_2$ on an asymptotic equivalence class $[\gamma]\in\partial\dl_2(2)$}
\label{actionfig}
\end{center}
\end{figure}

Notice that for any $[\gamma] \in (\partial T_0 -\omega_0)\times \omega_1$ and any $g\in L_2$, we will have $g\cdot[\gamma] \in (\partial T_0 -\omega_0)\times \omega_1$. Similarly, $(\partial T_1 -\omega_1)\times \omega_0$ is also invariant under the action of $L_2$.

\begin{observation}
\label{genactionobs}

The action of the generators $t$ and $at$ on the lamp stand model for $\partial\dl_2(2)$ is as follows:\begin{itemize}
	\item $t$ shifts the lit lamps one spot to the right (i.e. towards $+\infty$)
	\item $t^{-1}$ shifts the lit lamps one spot to the left (i.e. towards $-\infty$)
	\item $at$ shifts the lamps one spot to the right and then switches the lamp located at 0.
	\item $(at)^{-1}$ switches the lamp located at 0 and then shifts the lamps one spot to the left.
\end{itemize}

\end{observation}

For $k\in\integers$, let $a_k$ represent the element $t^kat^{-k}\in L_2$. Notice that in the lamps model for $L_2$, this is the element associated with only lamp $k$ lit and the lamplighter at position 0.

\begin{observation}

The action of $a_k$ on the lamps model of $\partial\dl_2(2)$ is to switch the lamp at position $k$.

\end{observation}

In Section \ref{dynamicssub}, we use the lamp stand interpretation of $\partial\dl_2(2)$ to compute the dynamics of this action.

\subsection{$\partial \dl_2(q)$ without a basepoint}
In Section \ref{visualboundariessub} we introduced the {\em based} and
{\em unbased} visual boundaries.  When $X$ is CAT(0) or
$\delta$-hyperbolic, these agree.  The following shows that the same
is true for $\dl_2(q)$.

\begin{proposition}
\label{basepointprop}

Let $\gamma$ be a geodesic ray in $\dl_2(q)$. 
Then there exists a geodesic ray $\tau$ emanating from the origin which 
is asymptotic to $\gamma$. 

\end{proposition}

\begin{proof}

In one tree $T_i$,
$i \in \{0,1\}$, $\gamma$ chooses a non-distinquished end $e \neq \omega_i$, and 
in $T_{1-i}$, $\gamma$ approaches $\omega_{1-i}$. Let $\tau$ be any
geodesic ray emanating from $o$ that approaches $e \in T_i$ and
$\omega_{1-i} \in T_{1-i}$.  Because the projections $\gamma^{(i)}$ and
$\tau^{(i)}$ approach the same end of $T_i$, they must merge since $T_i$ is a tree. I.e.,
there exist $r_1,r_2 \in \integers$ such that $\gamma^{(i)}(r_1) =
\tau^{(i)}(r_2)$. Similarly, there are $s_1,s_2 \in \integers$ such that
$\gamma^{(1-i)}(s_1) = \tau^{(1-i)}(s_2)$. Setting $n =
\max\{r_1,r_2,s_1,s_2\}$, one of $\gamma^{(i)}(k)$ and $\tau^{(i)}(k)$ is an 
ancestor of the other for all $k \geq n$, and the
opposite relation holds for $\gamma^{(1-i)}(k)$ and $\tau^{(1-i)}(k)$.
The distance from $\gamma(k)$ to $\tau(k)$ is constant, regardless of
$k$, and so the rays are asymptotic.
\end{proof}

\section{Topology of $\partial \dl_2(q)$}
\label{dl2qtop-section}

\subsection{Some important sets}

The natural topology on the visual boundary of a space is the topology
of uniform convergence on compact sets. Informally, this means that
two asymptotic equivalence classes are close if there are
representatives of those classes that share a long initial segment.
More formally, given a ray $\gamma$, a compact subset $[0,k]$ of
$[0,\infty)$ and $0<\epsilon<1$, define the set
$$B_{[0,k]}(\gamma,\epsilon)=\{ \gamma' \ | \ \sup\{d(\gamma(x),
\gamma'(x))\ |\ x\in [0,k]\} < \epsilon\}.$$ The sets
$B_{[0,k]}(\gamma,\epsilon)$ form a basis for the topology on the set of
geodesic rays. Often in our proofs,
we will work with representatives in the space of rays, rather than
the equivalence classes themselves. We will denote the equivalence
class of a ray $\gamma$ by $[\gamma]$. Abusing notation, we will write
$B_{[0,k]}([\gamma],\epsilon)$ for the image of $B_{[0,k]}(\gamma, \epsilon)$ in
the quotient space.

\begin{observation}
\label{obsquotientbasis}
The sets $B_{[0,k]}([\gamma],\epsilon)$ form a basis for the topology on the visual boundary
(the set of \emph{equivalence classes} of rays).
\end{observation}

\begin{definition} For $i\in\{0, 1\}$ and $n\in\mathbb{N}$, we define $C_n^i$ to be the set of equivalence classes of geodesic rays that ``bottom out'' in $T_i$ after descending for exactly $n$ edges. We define $C_0^i$ to be the set of equivalence classes of rays that ascend forever in $T_i$ without ever turning.
\end{definition}

Notice that when equipped with the subspace topology, the sets $C_n^i$ are homeomorphic to the Cantor set.

In terms of the lamp stand, elements of $C_n^0$ (for $n>0$) have a lit
lamp at position $-n$, no lit lamps below that position, and the
lamplighter at $+\infty$. Similarly, elements of $C_n^1$ (for $n>0$)
have a lit lamp at position $n-1$, no lit lamps above that position,
and the lamp lighter at $-\infty$. The lamp stand for an element of
$C_0^0$ has the lamplighter at $+\infty$ and no lamps lit below 0. The
lamp stand for an element of $C_0^1$ has the lamplighter at $-\infty$
and no lit lamps above -1.

\begin{definition}For $k\in\mathbb{N}$, we define the set $C_{k,\infty}^i = \cup_{n=k}^\infty C_n^i$, which is the set of equivalence classes of geodesic rays that descend at least $k$ edges in $T_i$ before turning and ascending in $T_i$ forever. 
\end{definition}

When equipped with the subspace topology, the sets $C_{k,\infty}^i$ are homeomorphic to the punctured Cantor set.

We can use these sets to better understand the topology on $\partial\dl_2(q)$.

\begin{observation}
\label{basisobs}

If $[\gamma]\in C_n^i$ for $n>0$ and $k\leq n$, then
$$B_{[0,k]}([\gamma],\epsilon) = C_{k,\infty}^{i} \cup C_0^{1-i}.$$

If $[\gamma]\in C_n^i$ for $n>0$ and $k>n$, then $$B_{[0,k]}([\gamma],\epsilon)
= \{[\tau]\in C_n^i \ | \ \tau^{(i)} \text{ agrees with }\gamma^{(i)}\text{ on
}[0,k]\}.$$

If $[\gamma]\in C_0^i$, then $$B_{[0,k]}([\gamma],\epsilon) = C_{k,\infty}^{1-i}
\cup\{[\tau]\in C_0^i \ | \ \tau^{(i)} \text{ agrees with }\gamma^{(i)}\text{ on
}[0,k]\}.$$

\end{observation}

\begin{proof}

These statements are easily verified; recall that $0<\epsilon<1$.

\end{proof}

See Figure \ref{nestedfig} for some examples of nested basis elements.

\begin{figure}[htbp]
\begin{center}
\includegraphics[width=3.5in]{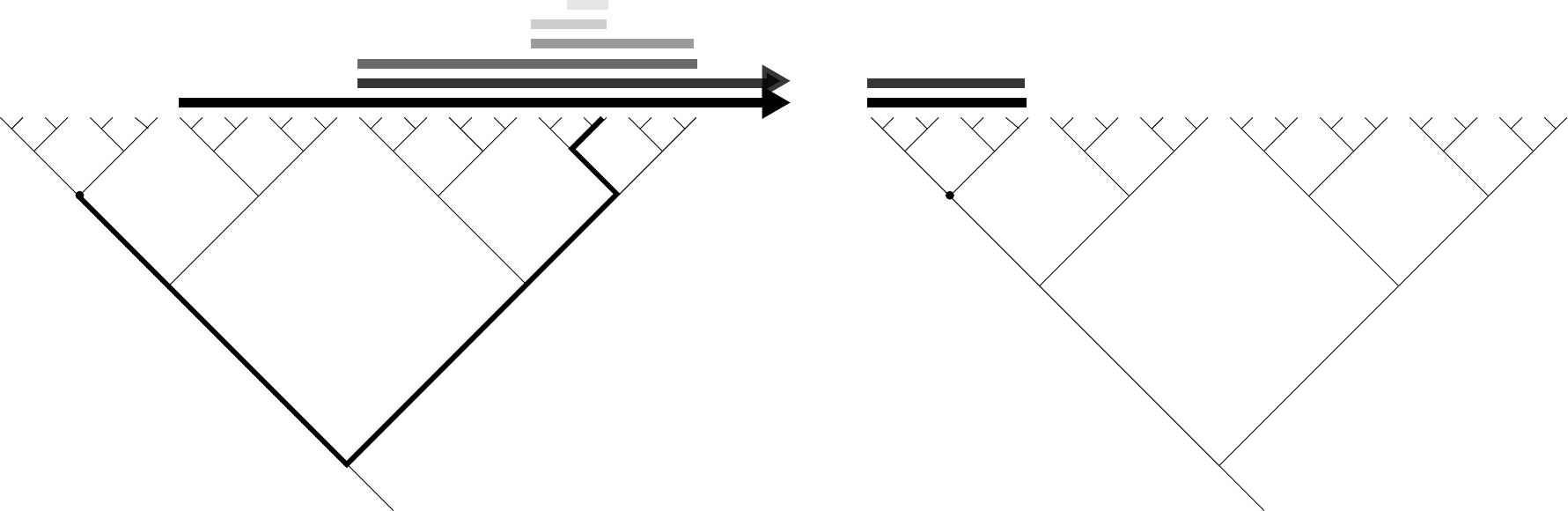}
\caption{Some nested basis elements of $\partial\dl_2(2)$}
\label{nestedfig}
\end{center}
\end{figure}

We now prove some of the important properties of these sets.

\begin{lemma}
\label{openobs}

For $n>0$, the set $C_n^i$ is open in $\partial \dl_2(q)$.

\end{lemma}

\begin{proof}

Fix $n>0$. For each $j\in\{1, 2, ... , q-1\}$, let $[\gamma_j]\in C_n^i$
such that if $j\neq j'$ then $\gamma_j^{(i)}(n+1) \neq
\gamma_{j'}^{(i)}(n+1)$ (note that $\gamma_j^{(i)}|_{[0,n]} =
\gamma_{j'}^{(i)}|_{[0,n]}$).

For $0<\epsilon<1$, notice that $$C_n^i = \bigcup_{j=1}^{q-1} B_{[0,n+1]}([\gamma_j], \epsilon).$$ Thus, $C_n^i$ is open.

\end{proof}

Lemma \ref{openobs} does not apply when $n=0$ because in this
case the open sets $B_{[0,1]}([\gamma_j], \epsilon)$  include 
all elements of $C^{1-i}_{1, \infty}$ (i.e. every class that bottoms out in the opposite tree). 
Hence $C_0^i$ cannot be
formed as a union in the same way. 

\begin{lemma}
\label{notopenobs}

The set $C_0^i$ is not open.

\end{lemma}

\begin{proof}

Fix $[\gamma] \in C_0^i$. For each $n>0$, let $\gamma_n$ be a ray that agrees with $\gamma$ on the first $n$ edges, but then bottoms out in $T_{1-i}$ and ascends in $T_{1-i}$ forever. In other words, $\gamma(x) = \gamma_n(x)$ for all $x\in[0,n]$ and $[\gamma_n] \in C_n^{1-i}$. Notice that $[\gamma_n] \notin C_0^i$.

Consider a basis element $B_{[0,k]}(\gamma, \epsilon)$ of the pre-quotient topology. If $n>k$, then $\gamma_n \in B_{[0,k]}(\gamma, \epsilon)$ and thus $[\gamma_n] \in B_{[0,k]}([\gamma], \epsilon)$. Thus, $[\gamma]$ is a limit point of $\{[\gamma_n]\}$, and so the complement of $C_0^i$ is not closed. Hence, $C_0^i$ is not open.

\end{proof}

\begin{observation}

For any $k\in \mathbb{N}$, the set $C_0^i \cup C_{k,\infty}^{1-i}$ is open.

\end{observation}

\begin{proof}


This follows directly from Observation \ref{basisobs}.

\end{proof}

\begin{observation}

For $n\geq0$, the set $C_n^i$ is closed.

\end{observation}

\begin{proof}

The complement is open by the previous observations.

\end{proof}

\subsection{Separability}

The boundary $\partial \dl_2(q)$ has some interesting separability properties that distinguish it from visual boundaries of hyperbolic or CAT(0) spaces.

\begin{definition} \cite[\S2.6]{munkres}

A topological space $X$ is $T_1$ if for every pair of points $x,y\in X$, there exist open sets $O_x, O_y$ such that $x\in O_x, y\notin O_x$ and $y\in O_y, x\notin O_y$.

\end{definition}

This is a weaker form of separability than the Hausdorff condition (also known as $T_2$), which requires that the open sets $O_x, O_y$ be disjoint.

\begin{observation}
\label{nonhausdorffobs}
The visual boundary $\partial \dl_2(q)$ is not Hausorff.

\end{observation}

\begin{proof} 

Let $\gamma$ and $\gamma'$ be distinct geodesic rays that ascend forever in $T_i$ with no turns (i.e. $[\gamma],[\gamma']\in C_0^i$). Notice that $[\gamma]\neq[\gamma']$. For each $n>0$, let $[\gamma_n]\in C_n^{1-i}$ be as in the proof of Lemma \ref{notopenobs}; that is, $\gamma_n$ agrees with $\gamma$ on the first $n$ edges before bottoming out in $T_{1-i}$. Notice that in the asymptotic equivalence class of $\gamma_n$, there is an element $\gamma'_n$ that agrees with $\gamma'$ on the first $n$ edges before bottoming out in tree $T_{1-i}$.

Thus, $[\gamma]$ and $[\gamma']$ are distinct limit points of the sequence $\{[\gamma_n]\}=\{[\gamma'_n]\}$, and so the topology is not Hausdorff.

\end{proof}

We could also prove that the topology is not Hausdorff using the following observation:

\begin{observation}
\label{gluingobs}

Any open set containing an element of $C_0^i$ necessarily contains $C_{k,\infty}^{1-i}$ for some $k$.

\end{observation}

\begin{proof}

This follows directly from Observation \ref{basisobs}.

\end{proof}

\begin{observation}
\label{t1obs}
The visual boundary $\partial \dl_2(q)$ is $T_1$.

\end{observation}

\begin{proof}

Let $\gamma$ and $\gamma'$ be geodesic rays that are not asymptotic to each other. So there exists some $k\in\mathbb{N}$ such that $\gamma(n) \neq \gamma'(n)$ for all $n\geq k$. 
Consider the basis elements $B_{[0,k]}([\gamma], \epsilon)$ and $B_{[0,k]}([\gamma'], \epsilon)$. For any ray $\tilde{\gamma}\in[\gamma]$, notice that $\tilde{\gamma}(k)\neq \gamma'(k)$, so $d(\gamma(k),\gamma'(k))\geq1>\epsilon$, and $[\gamma]\notin B_{[0,k]}([\gamma'], \epsilon)$. By symmetry, the reverse holds as well.

\end{proof}

\subsection{Compactness}

For $X$ non-positively curved, $\partial X$ is homeomorphic to the horofunction boundary of $X$ and is also an inverse limit of compact sets \cite[\S II.8]{bh}, both of which imply compactness. Since $\dl_2(q)$ is not CAT(0) or unique geodesic, we have to prove compactness directly.

\begin{proposition}
\label{compactobs}
$\partial \dl_2(q)$ is compact.

\end{proposition}

\begin{proof}

Let $\mathcal{A} = \{A_i\}_{i\in I}$ for some index set $I$ be an open cover of $\partial \dl_2(q)$. Without loss of generality, we may assume that these open sets are basis elements.

As sets, $\partial T_0 \sqcup \partial T_1 = \partial \dl_2(q) \sqcup \{\omega_0, \omega_1\}$. 
We extend $\mathcal{A}$ to a cover $\bar{\mathcal{A}}$ of $\partial T_0 \sqcup \partial T_1$ 
by defining
\[\bar{A_i} = A_i \cup \{ x \mid x = \omega_j \text{ for } j=0,1 \text{ and } C_{k,\infty}^{j} \subseteq A_i \text{ for } k>0\}. \]
Since $\mathcal{A}$ covers $\partial \dl_2(q)$,  
Observation \ref{gluingobs} ensures $\bar{\mathcal{A}}$ covers $\partial T_0 \sqcup \partial T_1$ (one can also see this from Observation \ref{basisobs} since our sets $A_i$ are assumed to be basis elements).

We now define covers $\bar{\mathcal{A}}^0$ and $\bar{\mathcal{A}}^1$ of $\partial T_{0}$ and  $\partial T_{1}$, respectively by $\bar{A}_i^0 = \bar{A}_i\cap\partial T_0$ and $\bar{A}_i^1 = \bar{A}_i\cap\partial T_1$. Since the $A_i$ are basis elements, it is easy to verify from Observation \ref{basisobs} that these projections $\bar{A}_i^0$ and $\bar{A}_i^1$ are open sets in the boundaries of the trees.

 
Thus, the covers $\bar{\mathcal{A}}^0,\bar{\mathcal{A}}^1$ are open. Since $\partial T_{0}$ and  $\partial T_{1}$ are compact, there exists a finite $F\subset I$ such that $\{\bar{A}_i^0\}_{i\in F}$ covers $\partial T_{0}$ and $\{\bar{A}_i^1\}_{i\in F}$ covers $\partial T_{1}$. Then $\{A_i\}_{i\in F}$ is a finite subcover of $\partial \dl_2(q)$.

\end{proof}

\subsection{Connectedness}

We have been considering the visual boundary through the Cantor sets $C_n^i$ and punctured Cantor sets $C_{k,\infty}^i$, so it is reasonable to expect that the visual boundary is disconnected in a similar manner to a Cantor set.

\begin{observation}
\label{disconnectedobs}

$\partial \dl_2(q)$ is totally disconnected.

\end{observation}

\begin{proof}

Let $S$ be a subset of $\partial \dl_2(q)$ containing at least two elements.

Suppose that $S\cap C_n^i\neq \varnothing$ for some $n>0$ and some $i\in\{0,1\}$. If $S\subseteq C_n^i$, then $S$ is disconnected since $C_n^i$ is a Cantor set. Else, since $C_n^i$ is both open and closed, it and its complement form a separation of $S$.

If $S\cap C_n^i = \varnothing$ for all $n>0$ and $i\in\{0,1\}$, then $S\subseteq \left( C_0^0\cup C_0^1\right)$. So $S$ is a subset of a Cantor set, and thus is disconnected.

\end{proof}

\subsection{Intuitive picture of the topology of $\partial\dl_2(q)$}

Intuitively, the visual boundary $\partial\dl_2(q)$ can be viewed as a
pair of punctured Cantor sets in which the punctures are ``filled'' by
a portion of the other Cantor set. Specifically, every open
neighborhood of $\omega_i$ becomes an open neighborhood of
$C_0^{1-i}$. Figure \ref{informalfig} illustrates this notion. 
Clearly, $\partial \dl_2(q)$ is not homogeneous.

\begin{figure}[htbp]
\begin{center}
\includegraphics[width=3in]{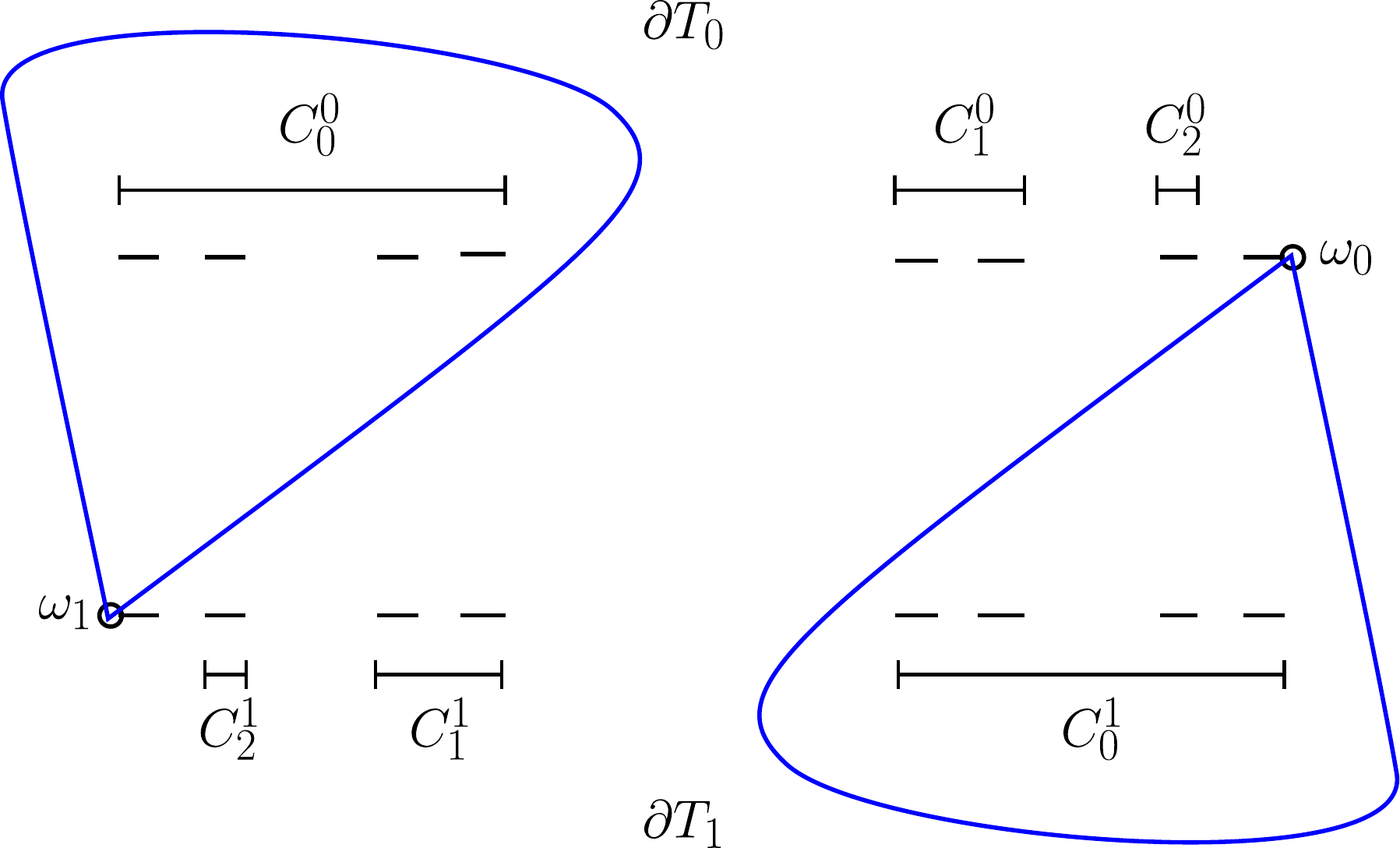}
\caption{An informal visualization of $\partial\dl_2(2)$}
\label{informalfig}
\end{center}
\end{figure}

\subsection{Dynamics of the action of $L_2$ on $\partial\dl_2(2)$}
\label{dynamicssub}

Notice that the exponent sum of $t$ in a word representing an element
$g$ of $L_2$ is equal to the position of the lamplighter in the lamp
stand representation of $g$. Thus, the exponent sum is an invariant of
the group element. In Section \ref{lamplighterbackgroundsubsection},
we
defined the function $exp_t(g)$ to denote the exponent sum of $t$ for $g$.

Notice that for an element $g\in L_2$, if $exp_t(g)=0$, then $g^2$ is trivial (since the second application of $g$ will switch off all the lights that the first application of $g$ switched on). Notice also that if $exp_t(g)\neq 0$, then $g$ will have infinite order since the lamplighter for $g^n$ with  $n\neq0$ is at position $n\cdot exp_t(g)\neq 0$. In other words, for $g\in L_2$ non-trivial, then the order of $g$ is either 2 (when $exp_t(g)=0$) or infinite (when $exp_t(g)\neq 0$).

\begin{definition}

Let $g\in L_2$ with $exp_t(g)>0$. Its lamp stand has no lit lamps below some position $m$. Consider the lamp stand  for $g^n$ for $n\in\mathbb{N}$. The lamplighter for $g^n$ is at position $n\cdot exp_t(g)$ and no matter how many more times we multiply by $g$, the lamps below position $n\cdot exp_t(g)+m$ will not be switched again. Thus, since $n\cdot exp_t(g)+m\to\infty$ as $n\to\infty$, we have a well-defined lamp stand for $g^\infty$. This lamp stand can be realized by a geodesic ray in $\dl_2(2)$ (since it is the Cayley graph of $L_2$) by starting with $t^m(at)$ and then multiplying by $t$ or $at$ for each successive lamp, depending on whether the lamp is lit or unlit in $g^\infty$. Thus, $g^\infty$ is an element of $\partial\dl_2(2)$.

We can similarly define $g^\infty$ for $g$ with $exp_t(g)<0$ (except that it will have no lit lamps \emph{above} $m$).

\end{definition}

Intuitively, $g^\infty$ is the ``lamp stand limit'' of $g^n$. For example, $t^\infty$ is the lamp stand with no lit lamps and the lamplighter at $+\infty$.

\begin{definition}

For $g\in L_2$ with $exp_t(g)\neq0$, we define $g^{-\infty}$ to be $(g^{-1})^\infty$.

\end{definition}

\begin{theorem}
\label{actionthm}

If a non-trivial element $g$ of $L_2$ has $exp_t(g) =0$, then its action on $\partial \dl_2(2)$ will be periodic of order 2. Otherwise, $g$ will act with north-south dynamics on the boundary, with the attractor in $(\partial T_0 - \omega_0)\times \omega_1$ and the repeller in $(\partial T_1 - \omega_1)\times \omega_0$ if $exp_t(g)>0$ and vice versa if $exp_t(g)<0$.

\end{theorem}

\begin{proof}

If $exp_t(g)=0$, then the action of $g$ on an element of $\partial\dl_2(q)$ will simply be to switch a finite set of lamps (the ones that are lit in the lamp stand interpretation of $g$).

For $g\in L_2$ with $exp_t(g)>0$ and $[\gamma]\in\partial\dl_2(q)$, define $[\gamma_n]\in\partial\dl_2(q)$ to be $g^n\cdot[\gamma]$. Assume $[\gamma]$ (and thus $[\gamma_n]$ for all $n$) has the lamplighter at $+\infty$ (i.e. $[\gamma]\in(\partial T_0 - \omega_0)\times \omega_1$). Thus, there is a minimum lit lamp, say at position $m$, in the lamp stand for $[\gamma]$. Then in the representation for $g^n\cdot[\gamma]$, all lamps at positions below $n\cdot exp_t(g)+m$ will be lit or unlit according to $g^n$'s lamp stand. For any $[0,k]\subseteq[0,\infty)$ compact and any $0<\epsilon<1$, let $n$ be large enough so that $n\cdot exp_t(g)+m>k$. Then notice that $[\gamma_n] \in B_{[0,k]}(g^\infty, \epsilon)$. Thus, $[\gamma_n] \to g^\infty$.

Similar arguments show the rest of the result.

\end{proof}

\begin{corollary}
\label{actioncor}

The action on $\partial\dl_2(2)$ of a non-torsion element of $L_2$ is hyperbolic.

\end{corollary}

\section{$\partial\dl_d(q)$ for $d>2$}
\label{dldqsection}

\subsection{Geodesics in $\dl_d(q)$}

Label each edge of each tree $T_i$ by 
an $\alpha \in \{0,\dots,q-1\}$, so that for each vertex $v \in T_i$ 
the edges moving up from $v$ correspond to $\{0,\dots, q-1\}$.
Modifying a concept from \cite{steintaback}, say an edge of $\dl_d(q)$ has
{\em type} $(i(\alpha) - j)$, $0\leq i\neq j < d$, $0\leq \alpha < q$, if it
ascends in $T_i$ along an edge labeled $\alpha$ and descends in $T_j$, or
$(i(\alpha) - j(\beta))$ if we wish to keep track of the descending label as
well. If
$i,j,k,l$ are pairwise distinct, any edge of type $(i(\alpha)-j)$ ``commutes''
with any edge of type $(k(\beta)-l)$, in the sense that, given an initial
vertex in $\dl_d(q)$, the two (uniquely
determined) paths of type 
$(i(\alpha)-j)(k(\beta)-l)$ and 
$(k(\beta)-l)(i(\alpha)-j)$ have the same terminal vertex. Moreover, 
$(i(\alpha)-j)$ commutes with $(k(\beta)-j)$. 
In addition, an adjacent pair of edges having
type $(i(\alpha)-j)(k(\beta)-i)$ can be replaced by the single edge of type
$(k(\beta)-j)$, since this pair creates an unnecessary backtrack in $T_i$.
Finally, $(j(\alpha)-i(\beta))(i(\beta')-k)$ can be
replaced with $(j(\alpha)-k)$ if and only if $\beta = \beta'$, as that is the only case
with backtracking.

This notation gives us a way to define turns in the $d>2$ case. 

\begin{definition}
\label{defturn}
A \emph{turn} in $T_i$ in a path in $\dl_d(q)$ is a subpath that
begins with an edge of type $(j(\alpha) -i)$ for some $j$ and $\alpha$, 
ends with an edge of type $(i(\beta) - k)$ for some $k$ and $\beta$, and 
no other edge type in the subpath involves $i$. 

In the case $d=2$, this is equivalent to our definition in 
Section \ref{asymptoticclassessection}. 

\end{definition}

We now show that any path whose
projection to $T_i$ turns back up the same edge is not geodesic. 

\begin{lemma}
\label{downthenimmediatelyuplemma}
Let $p$ be path in $\dl_d(q)$ following edges $e_1,e_2,\dots,e_n$ in order. Suppose for some
$0 \leq t < s \leq n$, $e_t$ is of type $(j(\alpha)-i(\beta))$, and $e_s$ is of 
type $(i(\beta)-k)$,
and all edges between $e_t$ and $e_s$ do not involve $T_i$. Then $p$ is not
geodesic.
\end{lemma}

\begin{proof}
We have a sequence of edge types 
\[ (j(\alpha)-i(\beta)) (a_0(\delta_0) - b_0)  \dots
(a_{s-t-1}(\delta_{s-t-1}) - b_{s-t-1}) (i(\beta) - k), \]
$a_x$, $b_x$, $i$ pairwise distinct, $0 \leq x < s-t$.  Using the commuting
relations discussed above, we may replace this subsequence with either 
\[ (j(\alpha)-i(\beta))(i(\beta)-k) (a_0(\delta_0)-b_0) \dots (a(\delta_{s-t-1})_{s-t-1} - b_{s-t-1}) \] 
if there is no $x$ with $a_x = k$, or if such an $x$ exists,  by 
\[ (j(\alpha)-i(\beta)) \dots (a_y(\delta_y) - b_y) (i(\beta)-k) \dots
(a_{s-t-1}(\delta_{s-t-1}) - b_{s-t-1}) \]
(where \mbox{$y = \max\{ x \mid a_x=k\}$}).
In either case, again by the preceding discussion, we may replace a two edge
subsequence  with
a single edge, without affecting the endpoints of the
subsequence. Hence, a shorter path is found, and $p$ is not geodesic.
\end{proof}

The following observation is trivial, but will be key in the proof of Theorem
\ref{geodesicturnsthm}. 

\begin{observation}
\label{straighteningobs}
Let $\pi$ be a path in $\dl_d(q)$ containing a subpath $\rho$ of
length $l$.
Consider the family $P_\rho$ of paths of length $l$ that begin at $\rho(0)$ and end at $\rho(l)$.
For any $\rho' \in P_\rho$, the path $\pi'$ constructed from $\pi$ by replacing
$\rho$ with $\rho'$ has the same initial and terminal vertices and the same
length as $\pi$.
{\raggedright \qed}
\end{observation}

\begin{theorem}
\label{geodesicturnsthm}
A geodesic in $\dl_d(q)$ has no more than one turn in each tree. 
\end{theorem}

\begin{proof}
Suppose a path $\pi$ of length $n$ has more than one turn in a tree $T_i$. Let $v_1, v_2 \in T_i$ be
the vertices where the first two turns in $T_i$ bottom out, in order. If $h_i(v_1) \geq h_i(v_2)$, 
then $\pi$ descends to $v_1$, turns, and must descend through $v_1$ again to
reach $v_2$. If $h_i(v_2) \geq h_i(v_1)$, then $\pi$ must 
ascend to $v_2$ (having passed through $v_1$), and in fact ascend above $v_2$, before it can turn at
$v_2$. Either way, there are $k,l \in \integers$ satisfying 
$0 < k < k+1 < l < n$, and $z \in \{1,2\}$ such that   $\pi^{(i)}(k) = v_z =
\pi^{(i)}(l)$. Moreover, by
assumption the
subpath $\rho$ = $\pi(k), \pi(k+1),\dots, \pi(l)$ has no turns in $T_i$, and has
at least one ascent in $T_i$ followed by at least one descent. So, applying
Observation  \ref{straighteningobs},  $\rho$
can be replaced  with a subpath $\rho'$ making the resultant path $\pi'$
satisfy Lemma \ref{downthenimmediatelyuplemma}. (To see this: when $h_i(v_1) \geq
h_i(v_2)$, we can choose $\rho'$ so that $\pi^{(i)}(k-1) = \rho'^{(i)}(k+1) =
\rho'^{(i)}(l-1)$; and if
$h_i(v_2) \geq h_i(v_1)$, then we can choose $\rho'$ so that $\rho'^{(i)}(k+1)
= \rho'^{(i)}(l-1) = \pi^{(i)}(l+1)$.) Hence $\pi'$ is not geodesic;
and since $\pi$ and $\pi'$ have the
same initial and terminal vertices and the same length, neither is $\pi$.
\end{proof}

\begin{corollary}
For any geodesic ray $\gamma$ in $\dl_d(q)$ and any $0\leq i < d$, 
 $\gamma^{(i)}$ approaches at most one end point of $T_i$,
and if $\gamma^{(i)}$ consists of finitely many edges, then $\gamma^{(i)}$ is
eventually constant in $T_i$. {\raggedright \qed}
\end{corollary}

\subsection{Asymptotic geodesic rays in $\dl_d(q)$}

We will show that $\partial \dl_d(q)$, $d>2$, has the indiscrete
topology.

\begin{definition}
We say that two geodesic rays \emph{have the same ends} if whenever one of them has a projection to a tree that has infinitely many edges, so does the other, and the two projections go to the same end of that tree.
\end{definition}

The visual boundary of $\dl_d(q)$, $d > 2$, will be significantly larger than
that of $\dl_2(q)$, as sets, not just because additional punctured
Cantor sets will be added for the trees, but also because it is no
longer guaranteed that two geodesic rays having the same ends will be 
asymptotic, due to the additional degree
of freedom offered by a third tree. However, since we aim
to show that when $d>2$ the boundary has the indiscrete topology, we
will not delve into this. We will show that any point of
$\partial \dl_d(q)$ is topologically indistinguishable\footnote{Two points 
are topologically indistinguishable from each other if every open set that 
contains one of these points contains the other as well.} 
from a point that
approaches a distinguished end $\omega_i$ in some $T_i$, a
nondistinguished end $e_j$, $j\neq i$, in some $T_j$, and is trivial
in every other tree. Thus, the above issue can be avoided in proving that
$\partial \dl_d(q)$ has the indiscrete topology.

\begin{observation}

\label{distinguishobs}

If $\tau_n$ are geodesic rays in $\dl_d(q)$ that are asymptotic to another geodesic ray $\gamma$ for all $n\in\mathbb{N}$, and $\tau$ is a geodesic ray that is a limit point of $\{\tau_n\}$ in the compact-open topology on geodesic rays before we quotient by the asymptotic equivalence classes, then $[\gamma]$ and $[\tau]$ are topologically indistinguishable elements of $\partial\dl_d(q)$.

\end{observation}

\begin{proof}

Clearly every neighborhood of $[\tau]$ contains $[\gamma]$ and since the basis definition of the topology is symmetric, every neighborhood of $[\gamma]$ contains $[\tau]$.

\end{proof}

\begin{lemma}

\label{finiteobs}

Let $\gamma$ be a geodesic ray in $\dl_d(q)$ for $d>2$. Partition
the set $\{T_0, T_1, T_2, \dots, T_{d-1}\}$ into sets $\mathcal{I}$
and $\mathcal{F}$, where the projection of $\gamma$ to any tree in 
$\mathcal{F}$ is eventually constant,  and the height of the projection 
to any tree in $\mathcal{I}$ approaches $\pm \infty$.
Then we can construct a geodesic ray $\tau$ that is asymptotic to
$\gamma$ and such that the projection of $\tau$ to any tree in $\mathcal{F}$ is
trivial. (Here trivial means the image is constant at the origin.)

\end{lemma}

\begin{proof}

Let $M$ be large enough that for each $T_i \in \mathcal{F}$, 
all edges of $\gamma$ that
project onto $T_i$ come before $M$, and for each tree in
$\mathcal{I}$ in which $\gamma$ bottoms out, $\gamma$ does so before $M$. 

Let $\rho$ be a geodesic ray such that for each tree $T_i$ in
$\mathcal{I}$, 
the projection $\rho^{(i)}$ approaches the same end as the projection 
$\gamma^{(i)}$, chosen so that the projection of $\rho$ to
any tree in $\mathcal{F}$ is trivial, and all of the turns in$\rho$ come 
before $N\geq M$. 

Let $\tau$ be defined by $\tau|_{[0,N]} = \rho|_{[0,N]}$, and for $n > N$, 
the $n$th edge of $\tau$ simply ``tracks'' the $n$th edge of
$\gamma$. That is, when the $n$th edge of $\gamma$ moves upward in some
$T_{i_n}$ and downward in some $T_{j_n}$, 
$\tau$ does the same, choosing the upward branch that takes it toward
the same point of $\partial T_{i_n}$ that $\gamma^{(i_n)}$ approaches. 
Since $\rho|_{[0,N]}$ is a geodesic and all turns in $\rho$ occur
before $N$, $\tau$ is a geodesic ray.

The ray $\tau$ has been chosen so that for each tree $T_i$, for $n
> N$, $d_{T_i}(\tau^{(i)}(n), \gamma^{(i)}(n))$ is constant. 
Lemma \ref{upperboundondistobs} then ensures $\tau$ and
$\gamma$ are asymptotic. 
\end{proof}

\begin{lemma}

\label{infiniteobs}

For $d>2$, let $\gamma$ be a geodesic ray in $\dl_d(q)$ with empty projection to $T_i$. Let $\tau$ be another geodesic ray whose projections in trees other than $T_i$ have the same ends as $\gamma$ and whose projection to $T_i$ is infinite. Then, $[\gamma]$ and $[\tau]$ are topologically indistinguishable.

\end{lemma}

\begin{proof}

Let $N$ be large enough so that all turns and finite projections of $\gamma$ and $\tau$ come before $N$. For $n>N$, define $\tau_n$ to be the ray that matches $\tau$ up through $n$ edges in $T_i$ and then ``tracks" $\gamma$ by going up and down in the same trees for each edge as in the proof of Lemma \ref{finiteobs}. By the same argument as in the proof of Lemma \ref{finiteobs}, each $\tau_n$ is asymptotic to $\gamma$. But clearly $\tau_n \to \tau$, so by Observation \ref{distinguishobs} we are done.

\end{proof}

\begin{corollary}

For $d>2$, an element of $\partial\dl_d(q)$ is topologically indistinguishable from at least one other element that only has non-empty projections in two trees: one that eventually ascends in height without bound, and the other which eventually descends in height without bound.

\end{corollary}

\begin{proof}

This follows immediately from Lemmas \ref{finiteobs} and \ref{infiniteobs}.

\end{proof}

\begin{lemma}

\label{switchobs}

Suppose that $\gamma$ is a geodesic ray in $\dl_d(q)$ for $d>2$ that has no projection to $T_i$ and infinite projection to $T_j$. Then we can construct a geodesic ray $\tau$ such that $\tau$ has no projection to $T_j$, $[\tau]$ is in every open set that contains $[\gamma]$, and $\gamma$ and $\tau$'s edges are exactly the same except that whenever $\gamma$ has an edge that would project to $T_j$, $\tau$ projects to the same exact edge in $T_i$.

In other words, $\gamma$ and $\tau$ are the same ray, just swapping the projections in $T_i$ and $T_j$ (one of which is empty), and the asymptotic equivalence classes of $\gamma$ and $\tau$ are topologically indistinguishable in $\partial \dl_d(q)$.

\end{lemma}

\begin{proof}

We begin by assuming that $\gamma^{(j)}$ eventually increases without bound in height. The descending case is analogous. Let $\ell$ be the number of edges that $\gamma^{(j)}$ descends before increasing forever.

In the obvious way, we can construct a geodesic ray $\tau$ that exactly matches $\gamma$, except that the infinite projection to $T_j$ and the empty projection to $T_i$ are swapped. We will now construct a sequence of geodesic rays $\tau_n$ such that $\tau_n\in[\gamma]$ and $\tau_n\to\tau$, which will show topological indistinguishability.

Let $N$ be large enough so that every turn and finite projection of $\gamma$  (and thus of $\tau$ also) occurs before $N$. For $n>N$, we construct $\tau_n$ as follows:

The first $n$ edges of $\tau_n$ exactly match the first $n$ edges of $\tau$. By choice of $N$, we have partitioned the trees into ``up,'' ``down,'' and ``empty'' for projections of $\tau$. That is, any edge after $N$ has its up projection in one of the ``up'' trees and its down projection in one of the ``down'' trees (after $N$ there are no edge projections in the ``empty'' trees). Notice that $T_j$ is ``empty'' for $\tau$, but is ``up'' for $\gamma$.

So for the next $\ell$ edges of $\tau_n$, continue to copy $\tau$, except that the ``down'' projections of the edges should all be in $T_j$ instead. By the definition of $\ell$, these down edges will exactly reach the point where $\gamma^{(j)}$ turns.

For all subsequent edges, $\tau_n$ ``mimics'' $\gamma$ by going up and down in the same trees as $\gamma$. As a result, for $x\geq n+\ell$, the distance between $\tau_n(x)$ and $\gamma(x)$ will be equal to the distance between $\tau_n(n+\ell)$ and $\gamma(n+\ell)$, so the two rays are in the same asymptotic equivalence class. Since $\tau_n \to \tau$, by Observation \ref{distinguishobs} we are done.

\end{proof}

\begin{corollary}
\label{updowncor}

For $d>2$, an element of $\partial\dl_d(q)$ is topologically indistinguishable from at least one other element that only has non-empty projections in trees $T_0$ and $T_1$ such that the projection to $T_0$ eventually ascends in height without bound, and the projection to $T_1$ eventually descends in height without bound.

\end{corollary}

\begin{theorem}
\label{indiscretethm}
For $d>2$, $\partial\dl_d(q)$ has the indiscrete topology.

\end{theorem}

\begin{proof}

Let $\gamma$ and $\gamma'$ be geodesic rays in $\dl_d(q)$. 
By Corollary \ref{updowncor}, we may assume that the only non-empty projections of $\gamma$ and $\gamma'$ are in trees $T_0$ and $T_1$, that $\gamma^{(0)}, \gamma'^{(0)}$ both eventually ascend in height without bound, and $\gamma^{(1)}, \gamma'^{(1)}$ both eventually descend in height without bound. Let $\ell$ (and $\ell'$) be the number of down edges in $\gamma^{(0)}$ (respectively, $\gamma'^{(0)}$) before turning.

For $N$ sufficiently large so that all the turns in $\gamma$ and $\gamma'$ occur after $N$ and for $n>N$, define $\tau_n$ so that the first $n$ edges go up always taking the leftmost edge in $T_2$ (recall $d>2$) and down in $T_1$. For the next $\ell$ edges, $\tau_n$ goes down in $T_0$ and up in $T_2$ (again, always taking the leftmost edge). After that, $\tau_n$ goes up in $T_0$ and down in $T_1$ towards the ends of $\gamma$. We define $\tau_n'$ similarly, but using $\ell'$ and $\gamma'$. Notice that $\tau_n|_{[0,n]} = \tau'_n|_{[0,n]}$.

The ray $\tau_n$ is asymptotic to $\gamma$, since the two rays are never further apart than their distance at $\tau_n(n+\ell)$ and $\gamma(n+\ell)$. Similarly, $\tau_n'$ is asymptotic to $\gamma'$.

But notice that for any $[0,k]\subseteq[0,\infty)$ and any $0<\epsilon<1$, we have $\tau_n \in B_{[0,k]}(\tau_n',\epsilon)$ for any $n\geq l$. Thus, $[\gamma]$ and $[\gamma']$ are topologically indistinguishable.

\end{proof}

\bibliographystyle{plain}
\bibliography{visualboundariestp}
\end{document}